\documentclass[11pt]{article}
\usepackage{latexsym,amsmath,amscd,amssymb,graphics,mathrsfs}
\usepackage{enumerate}
\usepackage{graphicx}
\usepackage{color}
\usepackage{soul}
\usepackage{hyperref}
\usepackage{comment}
\usepackage{geometry}
\usepackage{subcaption}
\usepackage{changepage}
\usepackage{url}
\usepackage{xcolor}
\usepackage{tabularx}
\usepackage{breakcites}

\newcommand{\rem}[1]{}
\makeatletter

\renewcommand\div{\on{div}}

\newcommand\on{\operatorname}

\newcommand\Emb{\on{Emb}}

\newcommand\Rot{\on{Rot}}

\newcommand\Diff{\on{Diff}}

\newcommand\ev{\on{ev}}

\newcommand\Ham{\on{Ham}}

\newcommand\curl{\on{curl}}
\renewcommand\prod{\on{prod}}

\newcommand\Gr{\on{Gr}}

\newcommand\OO{\mathcal{O}}

\newcommand\ZZ{\mathbb Z}

\newcommand\RR{\mathbb R}

\newcommand\w{\textsf{w}}
\newcommand\vv{\operatorname{v}}

\newenvironment{proof}[1][Proof]{\noindent\textbf{#1.} }{\ \rule{0.5em}{0.5em}}

\newcommand{\fint}{-\!\!\!\!\!\!\int}

\def\XXint#1#2#3{{\setbox0=\hbox{$#1{#2#3}{\int}$ }
\vcenter{\hbox{$#2#3$ }}\kern-.5\wd0}}

\begin{document}

\newtheorem{theorem}{Theorem}[section]
\newtheorem{definition}[theorem]{Definition}
\newtheorem{lemma}[theorem]{Lemma}
\newtheorem{remark}[theorem]{Remark}
\newtheorem{proposition}[theorem]{Proposition}
\newtheorem{corollary}[theorem]{Corollary}
\newtheorem{example}[theorem]{Example}

\definecolor{burgund}{RGB}{153,0,51}      


\title{Manifolds of vortex loops as coadjoint orbits}
\author{Ioana Ciuclea$^1$, Cornelia Vizman$^2$}
\addtocounter{footnote}{1}

\footnotetext{West University of Timi\c soara, 
300223-Timi\c soara, Romania.
\texttt{ioana.ciuclea@e-uvt.ro}
\addtocounter{footnote}{1}}

\footnotetext{West University of Timi\c soara, 
300223-Timi\c soara, Romania.
\texttt{cornelia.vizman@e-uvt.ro}
\addtocounter{footnote}{1}}

\date{ }
\maketitle
\makeatother



\begin{abstract}
We study a class of coadjoint orbits of the area preserving diffeomorphism group of
the plane consisting of vortex loops, namely closed curves in the plane decorated with
one-forms (vorticity densities) allowed to have zeros.
\end{abstract}

\noindent \textbf{Keywords:} \textcolor{black}{coadjoint orbit, vortex loop, Hamiltonian group}

\section{Introduction}

Euler's equations for an ideal fluid, $\dot \vv + \nabla_{\vv}\vv = -\nabla p$, $\div\vv = 0$, with $\vv$ denoting the fluid velocity and $p$ the pressure, are geodesic equations on the volume preserving diffeomorphism group \cite{Arnold, EM}. In particular, the vorticity $\curl \vv$ is confined to a coadjoint orbit of this group. In this article we focus on singular vorticities with 1-dimensional support in the two dimensional case. More precisely, we study a class of coadjoint orbits of the area preserving diffeomorphism group of the plane consisting of vortex loops, with the vorticity density allowed to have zeros. Here the area preserving diffeomorphism group is identified with $\Ham_c(\RR^2)$, the group of compactly supported Hamiltonian diffeomorphisms of $\RR^2$. Its Lie algebra can be seen as the space of compactly supported smooth functions on $\RR^2$, whose dual is the space of distributions. 

Vortex loops are the 2D version of vortex sheets (see \cite{batchelor, goldin}). These are pairs $(C, \beta_C)$ with $C$ denoting a closed plane curve and $\beta_C$ a volume form (its vorticity density):
\[\langle(C, \beta_C), h \rangle = \int_Ch\beta_C, \quad h \in  C^\infty_c(\RR^2). \]
We denote by $\OO_{\w}^a$ the set of vortex loops with the same enclosed area $a$ and the same total vorticity $\w = \int_C\beta_C$. This data forms a complete set of invariants of the coadjoint $\Ham_c(\RR^2)$ action on vortex loops. These coadjoint orbits are described in \cite{GBV} (see also \cite{CiucleaVizman}). 

\medskip

In this paper, we expand the study of vortex loops to include the case of a vorticity density with a finite number of nondegenerate zeros. These are pairs $(C, \beta_C)$ with $C$ an oriented closed plane curve and $\beta_C$ a one-form on $C$ with finitely many nondegenerate zeros (i.e. a Morse form). 
Choosing one zero as ``starting point'' allows us to regard them as a $k$-tuple $(x_1, \dots, x_k)$
and to introduce $k$ partial vorticities along the loop, $\w_i = \int_{x_i}^{x_{i+1}}\beta_C$,
regarded as a $k$-tuple $\bar\w=(\w_1, \dots, \w_k)$. We denote by $\OO_{\bar\w}^a$ the set of vortex loops with the same enclosed area $a$ and the same $\bar\w$, as this data forms a complete set of invariants for the coadjoint $\Ham_c(\RR^2)$ action on vortex loops. 

\medskip

Let $(S^1,\beta)$ be a model for the vortex loop $(C,\beta_C)$. It defines, together with the canonical area form $\omega$, a closed 2-form on the manifold of embeddings enclosing a fixed area $\Emb_a(S^1, \RR^2)$, 
\[\Omega_f(u_f, v_f) = \int_{S^1}\omega(u_f, v_f)\beta,  \]
which fails to be symplectic if $\beta$ is a volume form. In this case, the coadjoint orbit $\OO_{\bar\w}^a$ is identified with the quotient manifold $\Emb_a(S^1, \RR^2)/\Rot(S^1)$, which can be obtained via symplectic reduction in the Marsden-Weinstein ideal fluid dual pair \cite{MW}. 

\medskip

In this paper, we show that $\Omega$ is symplectic when 
$\beta$ is a Morse one-form (Proposition \ref{Omega_symp_ex}). When the partial vorticities have no rotational symmetry, the coadjoint orbit $\OO_{\bar\w}^a$ is identified with $\Emb_a(S^1, \RR^2)$ and $\Omega$ is its Kirillov-Kostant-Souriau (KKS) symplectic form (Remark \ref{orbit2}). When there is a (smallest) step $\ell$ rotational symmetry of $\bar\w$, the coadjoint orbit $\OO_{\bar\w}^a$ is identified with the quotient manifold $\Emb_a(S^1, \RR^2)/\mathbb Z_{k/\ell}$ (Corollary \ref{orbit1}), with $\ZZ_{k/\ell}$ denoting the (discrete) cyclic group of degree $k/\ell$ and the KKS symplectic form descending from $\Omega$. 

\medskip

The vortex loops considered here resemble the pointed vortex loops in \cite{CiucleaVizman}. These are singular vorticities that combine features of point vortices and vortex loops. More precisely, they are triples $\left(C, \beta_C, (x_1, \dots, x_k)\right)$ with $C$ a closed plane curve, $\beta_C$ a nowhere zero vorticity density and $(x_1, \dots, x_k)$ point vortices attached to the loop: 

\[\left\langle\left( C, \beta_C, (x_1, \dots, x_k)), h \right)\right\rangle = \int_Ch\beta_C + \sum_{i = 1}^k\Gamma_ih(x_i), \quad h \in C_c^\infty(\RR^2)\]
where $\Gamma_i$ denotes the vorticity of the point $x_i$. The point vortices introduce partial vorticities along the loop, namely $\w_i = \int_{x_i}^{x_{i+1}}\beta_C$. We denote by $\OO_{\bar\w}^a$ the set of pointed vortex loops with the same enclosed area $a$ and the same partial vorticities. The model for a pointed vortex loops is a triple $\left(S^1, \beta, (t_1, \dots, t_k)\right)$, with $(t_1, \dots, t_k)$ denoting points on the circle corresponding to the point vortices on the loop. The closed 2-form

\[\Omega^\Gamma = \Omega + \sum_{i = 1}^k\Gamma_i\ev_{t_i}^*\omega\]
where $\ev_{t_i}: \Emb_a(S^1, \RR^2) \to \RR^2$ is the evaluation map at the points $t_i$ of the circle is a symplectic form on $\Emb_a(S^1, \RR^2)$. If the vorticities $(\Gamma_i)$ and $(\w_i)$ have no joint rotational symmetry, the coadjoint orbit is identified with $\Emb_a(S^1, \RR^2)$ and $\Omega^\Gamma$ is its KKS symplectic form. If the vorticities $(\Gamma_i)$ and $(\w_i)$ have joint rotational symmetry of step $\ell$, the coadjoint orbit $\OO_{\bar\w}^a$ is identified with the quotient $\Emb_a(S^1, \RR^2)/\ZZ_{k/\ell}$, with KKS symplectic form descending from $\Omega^\Gamma$. 

We remark thus that the manifold $\OO_{\bar\w}^a$ of vortex loops enclosing a fixed area and featuring partial vorticites can be realized as two different coadjoint orbits of $\Ham_c(\RR^2)$, via identification with the quotient manifold $\Emb_a(S^1, \RR^2)/\ZZ_{k/\ell}$. If the partial vorticities arise from the existence of point vortices attached to the loop, the KKS symplectic form descends from $\Omega^\Gamma$. If the partial voticities arise because the vorticity density has isolated zeros, the KKS symplectic form descents from $\Omega$. In other words, the quotient manifold $\Emb_a(S^1, \RR^2)/\ZZ_{k/\ell}$ parameterizes two different coadjoint orbits of $\Ham_c(\RR^2)$ consisting of different types of vortex loops, with two different KKS symplectic forms. 

\medskip

\section{Manifolds of vortex loops}\label{vl}

We consider the pair $(S^1, \beta)$, where $S^1$ denotes the oriented unit circle and $\beta \in \Omega^1(S^1)$ is a one-form with nondegenerate zeros (i.e a Morse form). We denote by $Z(\beta)$ its zero set. As a consequence of the Morse lemma for a compact space, the zeros of $\beta$ are isolated and finitely many. By the Poincar\'e-Hopf lemma we deduce that the number of zeros of $\beta$ must be even.  Identifying $S^1$ with the quotient $\RR/2\pi\mathbb Z$ we obtain a natural ordering of the zeros: we may regard them as a $k$-tuple $\bar t = (t_1, \dots, t_k)$, with $0 \leq t_1 \leq \dots \leq t_k < 2\pi$. We make the following notations:
\begin{equation*}
\w := \int_{S^1}\beta, \quad \w_i := \int_{t_i}^{t_{i+1}}\beta,  
\end{equation*}
with the convention $t_{k+1} = t_1$. The $\w_i$'s form a $k$-tuple denoted by $\bar\w := (\w_1, \dots, \w_k)$, which inherits an ordering from $\bar t$. We remark that, because of the nondegeneracy of the zeros, the signs of the $\w_i$'s alternate. 

We consider vortex loops based on this model, namely pairs $(C, \beta_C)$ consisting of an oriented closed plane curve $C$, with no self intersections, endowed with a Morse vorticity density $\beta_C \in \Omega^1(C)$ (i.e. a one-form with nondegenerate zeros). We denote its zero set by $Z(\beta_C)$. Choosing a ``starting point'' $x_1 \in Z(\beta_C)$ allows us to denote the zeros of $\beta_C$ as a $k$-tuple $\bar x := (x_1, \dots, x_k)$ ordered with respect to the orientation of the curve. We denote the total vorticity by $\w := \int_C\beta_C$ and the partial vorticities induced by the zeros by $\w_i := \int_{x_i}^{x_{i+1}}\beta_C$, with the convention that $x_{k + 1} = x_1$. 

Let $f \in \Emb(S^1, \RR^2)$ be an embedding and $\beta \in \Omega^1(S^1)$ a Morse form with $\bar t = (t_1, \dots, t_k)$ and $\bar\w=(\w_1, \dots, \w_k)$. The diffeomorphism $f: S^1 \to f(S^1)$ induces a Morse form $f_*\beta$ on its image, with the same number of zeros as $\beta$ and the same set of $\w_i$'s. The map $t_i \mapsto f(t_i)$ induces an ordering of the zeros of $f_*\beta$, which in turn induces an ordering of the $\w_i$'s. 

\begin{remark}\label{ell}
Let $\ell \in \{1, \dots, k\}$ be the smallest natural number that satisifies 
\begin{equation*}\label{symmetry}
\w_i = \w_{i+\ell}, \quad \forall i \in\{1, \dots, k\}.
\end{equation*}
By convention, we consider $\w_{i+k} = \w_i$. We remark that $\ell$ must be a divisor of $k$. Because the numbers $\w_i$ alternate in sign, we remark that $\ell$ must also be an even number. If there exists such an $\ell < k$, we say that $\bar\w$ has rotational symmetry of step $\ell$. 
\end{remark}

\begin{definition}
If there exists $\ell \in \{1, \dots, k\}$ as described in Remark $\ref{ell}$, we say $\beta$  has rotational symmetry of step $\ell$. 
\end{definition}

Let $\Diff(S^1,\beta) = \{\gamma \in \Diff(S^1) : \gamma ^* \beta = \beta\}$ denote the stabilizer of $\beta \in \Omega^1(S^1)$ in $\Diff(S^1)$. Lemma \ref{isotropy} below shows that it is isomorphic to the cyclic group $\mathbb Z_{k/\ell}$.

\begin{remark}\label{def_j}
If $\gamma \in \Diff(S^1,\beta)$, then $\gamma(Z(\beta)) = Z(\beta)$, where $Z(\beta)$ denotes the zero set of $\beta$. Moreover, $\gamma$ cannot change the ordering of the zeros, so there exists $j$ such that $\gamma(t_i) = t_{i + j}$, for all $i = 1, \dots, k$ and $j$ is uniquely defined modulo $k$. Restricting to an interval between two consecutive zeros, we get a diffeomorphism $\gamma: [t_i, t_{i+1}] \to [t_{i+j}, t_{i+j+1}]$ which preserves $\beta$ and $\w_i = \int_{t_i}^{t_{i+1}}\beta = \int_{t_{i + j}}^{t_{i + j+1}}\beta = \w_{i+j}$, for all $i$. Thus $j$ must be a multiple of $\ell$. 
\end{remark}

\begin{lemma}\label{isotropy}
The group homomorphism 
\begin{equation}\label{phi}
g: \Diff(S^1,\beta) \to \mathbb Z_k, \quad g(\gamma) = j\mod k, 
\end{equation}
where $j$ is a natural number such that $\gamma(t_i) = t_{i+j}$, as described in Remark $\ref{def_j}$, is injective. Its image is the subgroup of $\ZZ_k$ generated by the element $\ell \mod k$ and is isomorphic to the cyclic group of degree $k/\ell$, namely $\mathbb Z_{k/\ell}$. 
\end{lemma}

\begin{proof}
Let $\gamma \in \text{Ker}(g)$, i.e. $\gamma(t_i) = t_i$, for all $i = 1, \dots, k$. On an interval between two consecutive zeros the form $\beta$ is a volume form, thus we have $\gamma: [t_i, t_{i+1}] \to [\gamma(t_i), \gamma(t_{i+1})]$ a volume preserving diffeomorphism. We compute $\int_{t_i}^t\beta = \int_{t_i}^t\gamma^*\beta =\int_{\gamma(t_i)}^{\gamma(t)}\beta = \int_{t_i}^{\gamma(t)}\beta$, which shows $\gamma(t) = t$ for all $t \in (t_i, t_{i+1})$. Therefore $\gamma$ is the identity on $S^1$ and $g$ is injective. 

We know that $j$ is a multiple of $\ell$, so it is enough to find $\gamma \in \Diff(S^1,\beta)$ with $g(\gamma) = \ell \mod k$. By the step $\ell$ rotational symmetry of $\bar\w = (\w_1, \dots, \w_k)$, we get that $\w_i = \w_{i+\ell}$ implies the existence of a $\beta$ preserving diffeomorphism $\gamma: [t_i, t_{i+1}] \to [t_{i+\ell}, t_{i + \ell + 1}]$. These define alltogether $\gamma \in \Diff(S^1,\beta)$. 
\end{proof}

\begin{proposition}\label{id}
Let $(S^1, \beta)$ denote the oriented circle together with a Morse form $\beta$ as described above. We denote by $\mathcal O_{\bar\w}$ the set of vortex loops with the same $\bar\w$ as $\beta$. If $\beta$ has rotational symmetry of step $\ell$, the cyclic group $\mathbb Z_{k/\ell}$ acts on $\Emb(S^1, \RR^2)$ via the isomorphism \eqref{phi} and the map
\begin{equation}\label{bij2}
\psi: \Emb(S^1, \RR^2)/\mathbb Z_{k/\ell} \to \mathcal O_{\bar\w}, \quad \psi([f]) = (f(S^1), f_*\beta)
\end{equation}
is a bijection.
\end{proposition}

\begin{proof}
Let $f_1, f_2\in\Emb(S^1, \RR^2)$ such that $\psi(f_1) = \psi(f_2)$. This means that $f_2 = f_1 \circ \gamma$, where $\gamma$ is a diffeomorphism of $S^1$ preserving the one-form $\beta$ which means $\gamma \in \Diff(S^1,\beta)$. By Lemma \ref{isotropy}, this isotropy group is isomorphic to $\mathbb Z_{k/\ell}$ hence factoring out this group ensures the injectivity of $\psi$. 

To prove the surjectivity, let $(C, \beta_C)$ be an element in $\mathcal O_{\bar\w}$ and $(x_1, \dots, x_k)$ an ordering of the zeros corresponding to the ordering of $\bar\w$. For every $\w_i > 0$, we construct diffeomorphisms $b_i: [t_i, t_{i+1}] \to [0, \w_i]$, $b_i(t) = \int_{t_i}^t\beta$ and $a_i: [x_i, x_{i+1}] \to [0, \w_i]$, $a_i(x) = \int_{x_i}^x\beta_C$. Using these maps, we construct a diffeomorphism $f_i: [t_i, t_{i+1}] \to [x_i, x_{i+1}]$, $f_i(t) = (a_i^{-1}\circ b_i)(t)$. We note that $db_i = \beta\big|_{[t_i, t_{i+1}]}$ and $da_i = \beta_C\big|_{[x_i, x_{i+1}]}$, which leads to $(f_i)_*\beta\big|_{[t_i, t_{i+1}]} = \beta_C\big|_{[x_i, x_{i+1}]}$. The construction for a $\w_i < 0$ is analogous, except for the codomain of $b_i$ and $a_i$ being $[\w_i, 0]$. The diffeomorphisms $f_i$, $i = 1, \dots, k$ glue together to define an embedding $f: S^1 \to \RR^2$, with $f(S^1) = C$ and $f_*\beta = \beta_C$. A different ordering of $(x_1, \dots, x_k)$ gives a different element in the class of $f$. We get $\psi([f]) = (f(S^1), f_*\beta)$. 

\end{proof}

\begin{corollary}
    If $\beta$ has no rotational symmetry, we have $\ell = k$ and the identification in \eqref{bij2} becomes 
    \begin{equation*}
        \psi: \Emb(S^1, \RR^2) \to \mathcal O_{\bar\w}, \quad \psi(f) = (f(S^1), f_*\beta).
    \end{equation*}
\end{corollary}

\begin{remark}[Decorated nonlinear Grassmannians framework \cite{HV_decorated}]

The nonlinear Grassmannian of type $S^1$ in $\RR^2$, which we denote by $\Gr_{S^1}(\RR^2)$, is the space of all submanifolds of $\RR^2$ which are diffeomorphic to $S^1$. It is the base space of the $\Diff(S^1)$ principal bundle
\begin{equation}\label{grass}
\pi: \Emb(S^1, \RR^2) \to \Gr_{S^1}(\RR^2), \quad \pi(f) = f(S^1).
\end{equation}

The space of vortex loops can be seen as a decorated Grassmannian, defined as an associated bundle to \eqref{grass}, namely
\begin{equation*}
\Gr_{S^1}^{\text{deco}}(\RR^2) = \Emb(S^1, \RR^2) \times_{\Diff(S^1)} \Omega^1(S^1).
\end{equation*}

We denote by $\Diff(S^1) \cdot \beta$ the orbit of $\beta \in \Diff(S^1)$ under the action of $\Diff(S^1)$. The space of vortex loops $\mathcal O_{\bar\w}$ can be seen as the decorated Grassmannian of type $(S^1, \beta)$, defined as an associated bundle to \eqref{grass}, namely
\begin{equation*}
\Gr_{(S^1, \beta)}^{\text{deco}}(\RR^2) = \Emb(S^1, \RR^2) \times_{\Diff(S^1)} \Diff(S^1)\cdot\beta.
\end{equation*}
Moreover, the map 
$
\Emb(S^1, \RR^2)     \rightarrow 	\Gr_{(S^1, \beta)}^{\text{deco}} (\RR^2)
$

is a principal bundle with structure group $\Diff(S^1,\beta)$. 

\end{remark}

\section{Symplectic structures for manifolds of embeddings}

Let $\mu \in \Omega^1(S^1)$ be a volume form on the circle and let $\omega$ be the canonical symplectic form in the plane. The hat pairing of $\omega$ and $\mu$ is defined as 
\begin{equation*}
\widehat{\omega\cdot\mu} = \fint_{S^1}\text{ev}^*\omega\wedge\text{pr}^*\mu,
\end{equation*}
where $\text{ev}: S^1 \times \Emb(S^1, \RR^2) \rightarrow \RR^2$ denotes the evaluation map and $\text{pr}: S^1 \times \Emb(S^1, \RR^2) \to S^1$ is the projection on the first factor. The two form $\widehat{\omega \cdot \mu}$ on $\Emb(S^1, \RR^2)$ is symplectic \cite{V}.

Let $\Emb_a(S^1, \RR^2)$ denote the manifold of embeddings enclosing a fixed area. The tangent space $T_f\Emb_a(S^1, \RR^2) = \left\{u_f : S^1 \rightarrow \RR^2 \mid \int_{S^1}f^*i_{u_f}\omega = 0\right\}$ can be identified with
\begin{equation}\label{tg_space}
T_f\Emb_a(S^1, \RR^2) \cong C_0^\infty(S^1) \times C^\infty(S^1)
\end{equation}
where $C_0^\infty(S^1) = \left\{\rho \in C^\infty(S^1) \mid \int_{S^1}\rho dt = 0\right\}$. Here we use the decomposition of vectors into tangent and normal components using the Euclidean metric in $\RR^2$.

\begin{lemma}
Let $\beta \in C^\infty(S^1)$ be a smooth function with nondegenerate zeros. The pairing $\langle \ , \ \rangle: C_0^\infty(S^1) \times C^\infty(S^1) \to \RR$ defined by
\begin{equation}\label{pairing}
\langle\rho, \lambda\rangle = \int_{S^1}\rho(t)\lambda(t)\beta(t)dt
\end{equation}
is nondegenerate.
\end{lemma}
\begin{proof}
Let $\rho\in C_0^\infty(S^1)$ such that 
\begin{equation}\label{rho}
\langle\rho,\lambda\rangle = 0, \quad \forall \lambda \in C^\infty(S^1).
\end{equation} 
We denote by $Z(\beta)$ the zero set of $\beta$. Assume by contradiction there exists $t \in S^1\setminus Z(\beta)$ such that $\rho(t) \neq 0$. Then there exists a neighborhood of $t$, say $U \subset S^1$, such that $\rho(t) \neq 0$, $\forall t \in U$. Without loss of generality, we can assume $\rho$ and $\beta$ are positive on $U$. Let $\lambda \in C^\infty(S^1)$ be a bump function with support in $U$, positive on $U$. Then 
\begin{equation*}
\langle\rho, \lambda\rangle = \int_{S^1}\rho(t)\lambda(t)\beta(t)dt > 0
\end{equation*}
which contradicts \eqref{rho}, therefore $\rho$ vanishes on $S^1\setminus Z(\beta)$. Since $\rho$ is a continuous function, this means $\rho = 0$. 

Furthermore, let $\lambda \in C^\infty(S^1)$ such that 
\begin{equation}\label{lambda}
\langle\rho, \lambda\rangle = 0, \quad \forall \rho \in C_0^\infty(S^1). 
\end{equation}
Consider two arbitrary points $t^\prime, t^{\prime\prime} \in S^1$ and let $\rho \in C_0^\infty(S^1)$ be a ``double bump'' zero integral function, supported in the union of small neighborhoods of $t^\prime$ and $t^{\prime\prime}$. By shrinking the support of $\rho$ we can make $\int_{S^1}\rho(t)\lambda(t)\beta(t)dt$ arbitrarily close to $\lambda(t^\prime)\beta(t^\prime) - \lambda(t^{\prime\prime})\beta(t^{\prime\prime})$. Using \eqref{lambda}, we get that $\lambda(t)\beta(t)$ is constant on $S^1$. Since $\beta(t) = 0$, for certain $t \in S^1$, we deduce that this constant must be 0. But $\beta(t) \neq 0$ for $t \in S^1 \setminus Z(\beta)$, therefore $\lambda(t) = 0$, $\forall t \in S^1\setminus Z(\beta)$. Since $\lambda$ is a continuous function, we get $\lambda = 0$. 
\end{proof}

Let $\beta \in \Omega^1(S^1)$ be an one-form with nondegenerate zeros. Using the procedure in \cite{V}, we define a two-form $\Omega \in \Omega^2(\Emb_a(S^1, \RR^2))$ by 
\begin{equation}\label{Omega}
\Omega_f(u_f, v_f) = (\widehat{\omega \cdot \beta})_f(u_f, v_f) = \int_{S^1}\omega(u_f, v_f)\beta. 
\end{equation}

\begin{proposition}\label{Omega_symp_ex}
The form $\Omega$ in \eqref{Omega} is symplectic and exact. 
\end{proposition}
\begin{proof}
Both $\omega$ and $\beta$ are closed, therefore $\Omega$ is also a closed form. The exterior derivative is a derivation for the hat pairing \cite{V}. This fact, together with $\omega = d\nu$ and $d\beta = 0$, leads us to 
\begin{equation*}
\Omega = \widehat{d\nu\cdot\beta} = d(\widehat{\nu\cdot\beta})
\end{equation*}
with $\widehat{\nu\cdot\beta}$ a one-form on $\Emb_a(S^1, \RR^2)$. 

Using the identification \eqref{tg_space}, we rewrite $\Omega$ as 
\begin{equation*}
\Omega((\rho_1, \lambda_1), (\rho_2, \lambda_2)) = \int_{S^1}\rho_1\lambda_2\beta - \rho_2\lambda_1\beta = \int_{S^1}\omega((\rho_1, \lambda_1), (\rho_2, \lambda_2))\beta
\end{equation*}
with $(\rho_1, \lambda_1), (\rho_2, \lambda_2) \in C_0^\infty(S^1) \times C^\infty(S^1)$ and remark that
\begin{equation*}
\Omega((\rho_1, \lambda_1), (\rho_2, \lambda_2))  = \langle\rho_1, \lambda_2\rangle - \langle\rho_2, \lambda_1\rangle
\end{equation*}
where $\langle, \rangle$ is the nondegenerate pairing \eqref{pairing}. 

\end{proof}

\section{Coadjoint orbits of the Hamiltonian group of $\RR^2$}

Euler's equations in the plane, $\dot\vv + \nabla_{\vv}\vv = -\nabla p$, $\text{div }\vv = 0$, for $\vv$ the fluid velocity and $p$ the pressure, are the geodesic equations on the Lie group of area preserving diffeomorphisms endowed with the right invariant $L^2$ metric \cite{Arnold}. The velocity $\vv$ is an element of the Lie algebra of this group, while the vorticity, $\text{curl } \vv$, is confined to a coadjoint orbit in the dual Lie algebra. Such orbits were classified in \cite{kirilov, IKM}.  Beside the smooth (regular) vorticities, non-smooth (singular) vorticities have also been considered, as these, too, are confined to coadjoint orbits (see \cite{K}). In this section, we describe coadjoint orbits consisting of vortex loops, i.e. singular vorticities with one-dimensional support. To do so, we will use the following proposition from mathematical folklore (a proof can be found in \cite{HV_nonlinear}):

\begin{proposition}\label{folklore}
Let $(M, \Omega)$ be a symplectic manifold and $G$ a Lie group, with Lie algebra denoted by $\mathfrak g$. Suppose the action of $G$ on $(M, \Omega)$ is transitive with injective equivariant momentum map $J: M \to \mathfrak g^*$. Then $J$ is one-to-one onto a coadjoint orbit of $G$. Moreover, it pulls back the Kirillov-Kostant-Souriau symplectic form $\omega_{\text{KKS}}$ on the coadjoint orbit to the symplectic form $\Omega$, i.e. $J^*\omega_{\text{KKS}} = \Omega$. 
\end{proposition}

Let $\omega = dx\wedge dy$ denote the canonical area form on $\RR^2$. This is a symplectic form which is also exact, for there exists $\nu = \frac{1}{2}(xdy - ydx)$ such that $\omega = d\nu$. We denote by $G$ the group of compactly supported area preserving diffeomorphisms. Its Lie algebra, which we denote by $\mathfrak g$, consists of compactly supported divergence free vector fields. In this setting, the group $G$ coincides with $\Ham_c(\RR^2)$, the group of compactly supported Hamiltonian diffeomorphisms of the plane and thus $\frak g$ is the Lie algebra of compactly supported Hamiltonian vector fields. This Lie algebra can be identified with $C_c^\infty(\RR^2)$, whose dual, denoted by $\mathfrak g^*$, is the space of distributions.

Recall that, by Proposition \ref{Omega_symp_ex}, $\Emb_a(S^1, \RR^2)$ is a symplectic manifold with symplectic form given by \eqref{Omega}.
The natural $G$ action on the symplectic manifold $(\Emb_a(S^1, \RR^2), \Omega)$, given by $\varphi \cdot f = \varphi \circ f$, is transitive (see for instance \cite{HV_weighted}). 

The hamiltonian $h \in C_c^\infty(\RR^2)$ and the one-form $\beta \in \Omega^1(S^1)$ induce a smooth function on $\Emb_a(S^1, \RR^2)$ via the hat pairing $\widehat{h \cdot \beta} : f \mapsto \int_{S^1}f^*h\beta$. Using the hat calculus in \cite{V}, we compute its diffe
ial: 

\begin{equation*}
d(\widehat{h \cdot \beta}) =  \widehat{(dh)\cdot\beta} + \widehat{h\cdot d\beta} = \widehat{\iota_{X_h}\omega\cdot\beta} = \iota_{\hat{X}_h}\Omega, 
\end{equation*}
where the vector field $\hat{X_h}(f) = X_h \circ f$ is the infinitesimal generator of $X_h \in \mathfrak g$. Therefore, the action of $G$ on $\Emb_a(S^1, \RR^2)$ is Hamiltonian with equivariant momentum map given by

\begin{equation}\label{momentum_map}
J: \Emb_a(S^1, \RR^2) \to \mathfrak g^*, \quad \langle J(f), X_h\rangle = \int_{S^1}(h \circ f)\beta.
\end{equation}
We will see how the injectivity of this momentum map depends on the rotational symmetry of $\beta$. 

The following Lemma for finite-dimensional manifolds also holds for Fr\'echet manifolds: 

\begin{lemma}\label{quotient}
    Let $(E, \Omega)$ be a symplectic Fr\'echet manifold. Let $G$ be an infinite dimensional Lie group and $H$ a discrete Lie group, both acting on $E$ such that the two actions commute. Suppose the symplectic form $\Omega$ is H invariant and the G action is Hamiltonian, with momentum map $J$. Then, the quotient $E/H$ is a symplectic manifold with symplectic form descending from $\Omega$ and the $G$ action on $E/H$ is Hamiltonian, with momentum map descending from $J$. 
\end{lemma}

The discrete group $\mathbb Z_{k/\ell}$ acts on $\Emb_a(S^1, \RR^2)$ via the isomorphism $\mathbb Z_{k/\ell} \cong \Diff(S^1,\beta)$ described in Lemma \ref{isotropy}. This action induces a diffeomorphism $\hat\gamma(f) = f \circ \gamma$, where $f\in \Emb_a(S^1, \RR^2)$ and $\gamma \in \Diff(S^1,\beta)$. Using the hat calculus in \cite{V}, we obtain 
\begin{equation*}
\hat \gamma^*\widehat{\omega \cdot \beta} = \widehat{\omega \cdot \gamma^*\beta} = \widehat{\omega \cdot \beta},
\end{equation*}
therefore the symplectic form $\Omega$ is preserved by the action. We remark that the actions of $\Diff(S^1)$ and $G$ commute, therefore, by Lemma \ref{quotient}, the symplectic form $\Omega$ descends to a symplectic form $\bar\Omega$ on the quotient $\Emb_a(S^1, \RR^2)/\mathbb Z_{k/\ell}$. Moreover, the action of $G$ on $\Emb_a(S^1, \RR^2)/\ZZ_{k/\ell}$ is Hamiltonian, with equivariant momentum map descending from $J$, given by 
\begin{equation}\label{momentum_map_quotient}
\bar J: \Emb_a(S^1, \RR^2)/\mathbb Z_{k/\ell} \to \mathfrak g^*, \quad \langle \bar J([f]), X_h\rangle = \int_{S^1}(h \circ f)\beta. 
\end{equation}

\begin{lemma}\label{inj_mommap}
The momentum map \eqref{momentum_map_quotient} is injective.
\end{lemma}

\begin{proof}

Let $f_1, f_2 \in \Emb_a(S^1, \RR^2)$ such that $\bar J([f_1]) = \bar J([f_2])$, which means $J(f_1) = J(f_2)$. Thus, 

\begin{equation}\label{inj}
\int_{S^1}(h \circ f_1)\beta = \int_{S^1}(h \circ f_2)\beta, \quad \forall h \in C_c^\infty(\RR^2).
\end{equation}

Assume by contradiction that $C_1 = f_1(S^1)$ and $C_2 = f_2(S^1)$ do not coincide. There exists a point $x \in C_1 \setminus C_2$, which means there exists a whole neighborhood of $x$ contained in $C_1$ but not in $C_2$. Shrinking this neighborhood, we get a subset $U \subset C_1\setminus C_2$ such that $f_1(t_i) \notin U$, for all $t_i \in Z(\beta)$, with $i = 1, \dots, k $. Let $h \in C_c^\infty(\RR^2)$ be a non-negative function supported in $U$. We get $\int_{S^1}(h \circ f_1)\beta \neq 0$ and $\int_{S^1}(h \circ f_2)\beta = 0$, which contradicts \eqref{inj}. The assumption has been false, so the embeddings have the same image, i.e. there exists $\gamma \in \Diff(S^1)$ (assumed to be orientation preserving) such that $f_2 = f_1 \circ \gamma$. The equation \eqref{inj} becomes
\begin{equation}\label{inj2}
\int_{S^1}(h \circ f_1)(\beta - \gamma_*\beta) = 0, \quad \forall h \in C_c^\infty(\RR^2). 
\end{equation}

Assume by contradiction $\beta - \gamma_*\beta \neq 0$. Then there exists a subset of $V \subset S^1$ where $\beta - \gamma_*\beta$ has constant sign. Let $h \in C_c^\infty(\RR^2)$ be a function such that $h \circ f_1$ is supported in $V$. This leads us to $\int_{S^1}(h \circ f_1)(\beta - \gamma_*\beta) \neq 0$, which contradicts \eqref{inj2}. The second assumption has been false, so $\beta = \gamma_*\beta$, i.e. $\gamma$ is an element in the isotropy subgroup $\Diff(S^1,\beta)$. By Lemma \ref{isotropy}, this subgroup is isomorphic to $\mathbb Z_{k/\ell}$. Factoring out this cyclic group ensures the injectivity of $J$. 
\end{proof}

Each vortex loop $(C, \beta_C) \in \mathcal O_{\bar\w}$ corresponds to a unique non-smooth element in the dual: 
\begin{equation*}
\langle(C, \beta_C), h \rangle = \int_Ch\beta_C, \quad h \in \mathfrak g \cong C^\infty_c(\RR^2). 
\end{equation*} 

The natural action of $G$
\begin{equation}\label{cdj_action}
\varphi \cdot (C, \beta_C) = (\varphi(C), \varphi_*\beta_C),
\end{equation}
leaves invariant the area enclosed by the curve as well as all partial vorticities $\w_i$. We denote by $\mathcal O_{\bar\w}^a$ the subset of $\mathcal O_{\bar\w}$ consisting of vortex loops enclosing a fixed area $a := \int_C\nu$. Using Proposition \ref{id}, we obtain an identification

\begin{equation}\label{id_emba}
    \bar \psi: \Emb_a(S^1, \RR^2)/\mathbb Z_{k/\ell} \to \mathcal O_{\bar\w}^a, \quad \bar\psi([f]) = (f(S^1), f_*\beta). 
\end{equation}
The bijection \eqref{id_emba} intertwines the action of $G$ on $\Emb_a(S^1, \RR^2)/\ZZ_{k/\ell}$ with the natural action \eqref{cdj_action}.

\begin{theorem} The quotient space $(\Emb_a(S^1, \RR^2)/\mathbb Z_{k/\ell}, \bar\Omega)$ defined in this section can be realized as a coadjoint orbit of $\Ham_c(\RR^2)$. More precisely, the momentum map $\bar J$ in \eqref{momentum_map_quotient} is one-to-one onto a coadjoint orbit and the KKS symplectic form satisfies $\bar J^* \omega_{\text{KKS}} = \bar\Omega$. 
\end{theorem}

\begin{proof}
The proof follows from Proposition \ref{folklore}, together with the transitivity result from \cite{HV_weighted} and Lemma \ref{inj_mommap}. 
\end{proof}

\begin{corollary}\label{orbit1}
    The manifold of vortex loops $\mathcal O_{\bar\w}^a$, identified with $\Emb_a(S^1, \RR^2)/\mathbb Z_{k/\ell}$ via \eqref{id_emba}, is one-to-one onto a coadjoint orbit of $\Ham_c(\RR^2)$. 
\end{corollary}

\begin{remark}\label{orbit2}
    If $\beta$ has no rotational symmetry, the group $\mathbb Z_{k/\ell}$ is trivial. We obtain that the manifold of vortex loops $\mathcal O_{\bar\w}^a$ is identified with the symplectic manifold $(\Emb_a(S^1, \RR^2), \Omega)$ which is one-to-one onto a coadjoint orbit of $\Ham_c(\RR^2)$ via the momentum map $J$ in \eqref{momentum_map}. 
\end{remark}

\textbf{Funding:} I.C. was financially supported by the START Grant of the West University of Timi\c soara during the writing of this paper.

\textbf{Acknowledgments:}
I.C was supported by the START Grant of the West University of Timi\c soara during the writing of this paper.

\end{document}